\documentclass[reqno]{amsart}

\usepackage{amsmath,amssymb,amscd, accents}
\usepackage{graphicx}
\DeclareGraphicsExtensions{.eps}
\usepackage{mathrsfs}
\usepackage[mathcal]{eucal}
\usepackage{color}

\newtheorem{Cor}{Corollary}{\bfseries}{\itshape}
\newtheorem{Thm}[Cor]{Theorem}{\bfseries}{\itshape}
\newtheorem*{Thm*}{Theorem}{\bfseries}{\itshape}
\newtheorem{Prop}[Cor]{Proposition}{\bfseries}{\itshape}
\newtheorem{Lem}[Cor]{Lemma}{\bfseries}{\itshape}
\newtheorem*{Lem*}{Lemma}{\bfseries}{\itshape}
\newtheorem{Fact}[Cor]{Fact}{\bfseries}{\itshape}
{\bfseries}{\itshape}
{\bfseries}{\rmfamily}
\newtheorem{Ex}[Cor]{Example}{\scshape}{\rmfamily}
\newtheorem{Rem}[Cor]{Remark}{\scshape}{\rmfamily}

\renewcommand\ge{\geqslant} \renewcommand\le{\leqslant}
\let\tildeaccent=\~ \let\hataccent=\^
\renewcommand\~[1]{\widetilde{#1}}

\def\<{\left<} \def\>{\right>} \def\({\left(} \def\){\right)}

\let\parasymbol=\S \def\secref#1{\parasymbol\ref{#1}}

\let\polishL=l \def\Zoladek.{\.Zol\c adek}

\def\codim{\operatorname{codim}}

 \def\etc.{\emph{etc}.}

\def\:{\colon} \def\R{{\mathbb R}} \def\C{{\mathbb C}}
\def\D{{\mathbb D}}
\def\A{{\mathbb A}}
\def\Z{{\mathbb
    Z}}  \def\Q{{\mathbb Q}} \def\P{{\mathbb P}}
 \def\e{\varepsilon} \def\S{\varSigma}

 \def\J{\operatorname{J}}

 \def\d{{\mathrm d}}

 \let\PolishL=\L \def\Lojas.{\PolishL ojasiewicz}
\def\cN{{\mathcal N}} 
  
\def\cF{{\mathcal F}}

\def\clo{\operatorname{Clo}}
\def\vol{\operatorname{Vol}}

\def\w{\omega}

\def\mv{V}

\def\^#1{^{(#1)}{}}
\def\J{\mathcal{J}}
\def\red#1{\mathcal{R}(#1)}
\def\oml{\Omega\^l}
\def\iso{\operatorname{Iso}}

\begin{document}

\title{Bezout-type theorems for differential fields}

\author{Gal Binyamini}\address{University of Toronto, Toronto, 
Canada}\email{galbin@gmail.com}
\thanks{The author was supported by the Banting
  Postdoctoral Fellowship and the Rothschild Fellowship}

\begin{abstract}
  We prove analogs of the Bezout and the
  Bernstein-Kushnirenko-Khovanskii theorems for systems of algebraic
  differential conditions over differentially closed fields. Namely,
  given a system of algebraic conditions on the first $l$ derivatives
  of an $n$-tuple of functions, which admits finitely many solutions,
  we show that the number of solutions is bounded by an appropriate
  constant (depending singly-exponentially on $n$ and $l$) times the
  volume of the Newton polytope of the set of conditions. This
  improves a doubly-exponential estimate due to Hrushovski and
  Pillay.

  We illustrate the application of our estimates in two diophantine
  contexts: to counting transcendental lattice points on algebraic
  subvarieties of semi-abelian varieties, following Hrushovski and
  Pillay; and to counting the number of intersections between isogeny
  classes of elliptic curves and algebraic varieties, following
  Freitag and Scanlon. In both cases we obtain bounds which are
  singly-exponential (improving the known doubly-exponential bounds)
  and which exhibit the natural asymptotic growth with respect to the
  degrees of the equations involved.
\end{abstract}

\maketitle
\date{\today}

\section{Introduction}

In its most elementary form, Bezout's theorem states that a subset of
$\C^n$ defined by equations $P_1,\ldots,P_n$ of respective degrees
$d_1,\ldots,d_n$ can have at most $d_1\cdots d_n$ \emph{isolated}
points. Various generalizations of this statement have been proposed.
For example, the Bernstein-Kushnirenko-Khovanskii theorem estimates
the number of isolated points in terms of the \emph{mixed volume} of
the Newton polytopes $\Delta(P_1),\ldots,\Delta(P_n)$. As a
consequence of the Bezout theorem and its generalizations, whenever a
set defined within the algebraic category happens to be finite, one
can produce effective estimates for the size of the set (which often
turn out to be fairly accurate).

In this paper we consider generalizations of Bezout's bound for
systems of differential equations. The fundamental question is as
follows: \emph{given a system of algebraic conditions on an $n$-tuple
  of functions and their first $l$ derivatives, which admits finitely
  many solutions, can one estimate the number of solutions in terms of
  the degrees of the equations involved?}

This question has been considered by Hrushovski and Pillay in
\cite{hp:effective}. Their result, quoted in Theorem~\ref{thm:hp-main}
below, is a powerful analog of the Bezout theorem which similarly
allows one to translate qualitative finiteness results obtained using
differential-algebraic and model theoretic methods into effective
estimates (two examples of a diophantine nature are discussed below).

The explicit estimate in Theorem~\ref{thm:hp-main} is stated in terms
of slightly different algebraic data than our naive formulation of the
Bezout bound, making it difficult to make a direct comparison. None
the less, it is clear that the Bezout bound has two advantages. First,
the Bezout bound (and its various generalizations) are
singly-exponential with respect to the ambient dimension, whereas
Theorem~\ref{thm:hp-main} admits doubly-exponential growth. Second,
assuming for example that all equations involved have degree $d$, the
Bezout bound is a polynomial in $d$ with exponent equal to the ambient
dimension. This is the natural asymptotic for solutions of
$n$-dimensional systems of equations, and is known of course to be
optimal in the algebraic context. The bound of
Theorem~\ref{thm:hp-main} depends polynomially on the degree as well,
but the exponent can be in general much larger than the ambient
dimension (indeed, exponentially large).

Our principal result is an estimate which recovers the two asymptotic
properties above in the differential algebraic context. The statements
are presented in~\secref{sec:results}, with formulations in terms of
degrees and, more generally, mixed volumes of Newton polytopes.

Theorem~\ref{thm:hp-main} has been applied to produce effective bounds
for some diophantine counting problems:
\begin{enumerate}
\item[i.] In \cite{hp:effective} Theorem~\ref{thm:hp-main} was used to
  derive bounds on the number of transcendental lattice points on
  algebraic subvarieties of semi-abelian varieties. In particular, in
  the case of the two-dimensional torus this allowed the authors to
  produce doubly-exponential bounds for a counting problem due to
  Bombieri, improving the previous repeated-exponential bounds
  obtained by Buium \cite{buium:torus}.
\item[ii.] More recently, Theorem~\ref{thm:hp-main} has been used in
  \cite{fs:j-func} to derive bounds on the number of intersections
  between isogeny classes of elliptic curves and algebraic varieties.
  In particular, in the two-dimensional case this allowed the authors
  to produce doubly-exponential bounds for a counting problem due to
  Mazur.
\end{enumerate}

Both of these applications follow the same principal strategy,
inspired by the work of Buium \cite{buium:lang,buium:effective-lang}
(see \cite{pillay:survey,scanlon:survey} for surveys). Namely, given a
counting problem of a diophantine nature (over $\C$, for example) one
first expands the structure by adjoining to $\C$ a differential
operator making it into a differentiably-closed field with field of
constants $\bar\Q$. One then writes a system of differential algebraic
conditions satisfied by the solutions of the diophantine problem.
Finally, one shows (and this is naturally the deeper step involving
arguments specific to the problem at hand) that the system, while
possibly admitting more solutions than the original diophantine
problem, still admits finitely many solutions.

After carrying out the steps above, one can use
Theorem~\ref{thm:hp-main} to produce an explicit estimate for the
number of solutions of the diophantine problem.
In~\secref{sec:applications} we apply our estimates to obtain refined
(singly-exponential) bounds for problem (i) and (ii) above.

\subsection{Setup and notations}

Let $K$ denote a differentially closed field of characteristic zero
with differential $D$, and $K_0$ its field of constants. Let $M$
denote an ambient space, which we will take to be either an affine
space $K^n$ or a Zariski open dense subset thereof. Denote by $\xi$ a
coordinate on $M$. We define the $l$-th prolongation as
$M\^l=M\times K^{nl}$, and denote by $\xi=\xi\^0,\xi\^1,\ldots,\xi\^l$
coordinates on $M\^l$. For simplicity of the notation we denote
$s:=\dim M\^l=n(l+1)$. We denote by $\pi_M:M\^l\to M$ the projection
onto $M$. If there is no ambiguity we omit $M$ from this notation.

For any $x\in M$, denote by $x\^l\in M\^l$ the $l$-jet of $x$, i.e.
the element $(x,\ldots,D^l x)$. We denote by $\J^l(M)\subset M\^l$
the set of all $l$-jets in $M\^l$,
\begin{equation}
  \J^l(M) := \{ x\^l : x\in M\}.
\end{equation}

Let $\oml\subset M\^l$ denote a Zariski open dense subset of $M\^l$
(we include $l$ in the notation $\oml$ in the interest of clarity, but
this is in fact an arbitrary open dense set). If we do not explicitly
specify $\oml$ then it is taken to be equal to $M\^l$. 

Let $Y\subset\oml$ be a Zariski closed set. We view $Y$ as a system
of algebraic-differential conditions on $M$. We say that $x\in M$ is a
\emph{solution for $Y$} if $x\^l\in Y$. We define the \emph{reduction of $Y$}
$\red Y\subset Y$ to be
\begin{equation}
  \red Y := \clo [\J^l(M)\cap Y]
\end{equation}
where $\clo$ denotes Zariski closure. In other words, the reduction of
$Y$ is the Zariski closure of the set of jets of the solutions of $Y$.
If $Y$ admits finitely many solutions then $\red Y$ is the set
consisting of their jets. In this case we denote the number of
solutions by $\cN(Y)$.

\begin{Ex}
  Suppose $M=K^2$ with coordinates $\xi,\eta$, and consider the set
  $Y\subset M\^1$ given by the conditions $\xi\eta=1$ and $\xi\^1\eta+\xi\eta\^1=1$.
  Then $\dim Y=2$. However, it is clear that $Y$ admits no solutions,
  since deriving the first condition we see that any solution must
  also satisfy $\xi\^1\eta+\xi\eta\^1=0$, contradicting the second condition.
  Thus $\red Y=\emptyset$.
\end{Ex}

In \cite{hp:effective}, Hrushovski and Pillay considered the problem
of estimating $\cN(Y)$ in terms of the algebraic degree of $Y$. In
particular, using differential algebraic methods they have established
the following estimate.

\begin{Thm}[\protect{\cite[Proposition~3.1]{hp:effective}}] \label{thm:hp-main}
  Let $X\subset M$ and $S\subset \oml$ be Zariski closed. Denote
  $m:=\dim X$. Denote $Y=S\cap\pi^{-1}(X)$ and suppose that $Y$ admits
  finitely many solutions. Then
  \begin{equation}
    \cN(Y) \le \deg(X)^{l2^{ml}} \deg(S)^{2^{ml}-1}.
  \end{equation}
\end{Thm}

\subsection{Statement of our results}
\label{sec:results}

Recall that for a polynomial $P\in K[\xi\^0,\ldots,\xi\^l]$, its
\emph{Newton polytope} denoted $\Delta(P)$ is defined to be the convex
hull of the set of exponents (of terms with non-zero coefficients) of
$P$, viewed as a subset of $\R^s$. We say that $\Delta$ is a
\emph{lattice polytope} if it is obtained as the Newton polytope of
some polynomial. We will denote by $\Delta_{\xi\^j}$ the standard
simplex in the $\xi\^j$ variables, and by $\Delta_\xi\^j$ the standard
simplex in the $\xi\^0,\ldots,\xi\^l$ variables.

The classical Bernstein-Kushnirenko-Khovanskii theorem (henceforth
BKK) relates the number of solutions of a system of polynomial
equations to the mixed volume of their Newton polytopes (we refer the
reader to \cite{bernstein:bk} for the notion of mixed volume).

\begin{Thm}[\protect{\cite{kushnirenko:bk,bernstein:bk}}]
  Let $\Delta_1,\ldots,\Delta_s$ be lattice polytopes in $\R^s$. Then
  for any tuple $P_1,\ldots,P_s$ with $\Delta(P_i)=\Delta_i$, the
  system of equations $P_1=\cdots=P_n=0$ admits at most $\mu$
  isolated solutions in $(K^*)^s$, where
  \begin{equation}
    \mu = n! V(\Delta_1,\ldots,\Delta_s)
  \end{equation}
  and $V(\cdots)$ denotes the mixed volume. Moreover, for a
  sufficiently generic choice of the tuple $P_i$ the bound $\mu$ is
  attained.
\end{Thm}

\begin{Rem}\label{rem:co-ideal-poly}
  Under some mild conditions, the BKK estimate also bounds the number
  of zeros in $K^s$. For instance, this is true if $\Delta$ is the
  convex full of any finite co-ideal $I\subset \Z_{\ge0}$ (since in
  this case one can always, after a generic translation, assume that
  none of the solutions lie on the coordinate axes). We will refer to
  such polytopes as \emph{co-ideal polytopes}. The common case of the
  polytope of polynomials with bounded degree or multi-degree
  certainly satisfies this condition.

  Our results can be applied either to $(K^*)^s$ or to $K^s$, but in
  the latter case we tacitly assume that all polytopes involved are
  co-ideal polytopes.
\end{Rem}

Our main result is an analog of the first part of the BKK theorem.
Namely, for a system of algebraic-differential conditions $Y$
admitting finitely many solution, we estimate the number of solutions
in terms of a mixed volume of the Newton polytopes associated to the
equations defining $Y$. Note that our various formulations are stated
in terms of the degrees of the equations defining $Y$, whereas
Theorem~\ref{thm:hp-main} is stated in terms of the algebraic degree
of $Y$ as a variety.

We begin with a formulation valid for complete intersections (for a
slightly more general form valid for flat limits of complete
intersections see Proposition~\ref{prop:flat-limit-bound}).

\begin{Thm}\label{thm:comp-intersections}
  Let $Y\subset\oml$ be a complete intersection defined by
  polynomials with Newton polytopes $\Delta_1,\ldots,\Delta_k$, and
  suppose that $\cN(Y)<\infty$. Denote
  \begin{equation}
    \Gamma := (s+1)\Delta_\xi+\Delta_1+\cdots+\Delta_k.
  \end{equation}
  Then
  \begin{equation}
    \cN(Y) \le C_{s,k} \mv(\Delta_1,\ldots,\Delta_k,\Gamma,\ldots,\Gamma)
    \qquad C_{s,k} = (s!) (\delta+2)^{\delta(\delta+1)/2}
  \end{equation}
  where $\delta:=s-k$.  
\end{Thm}

We note that the factor of $s!$ in $C_{s,k}$ is the usual factor appearing
in the BKK theorem. The additional factor is an artifact of our construction
and could certainly be improved somewhat.

We next present a result valid for arbitrary varieties rather than complete
intersections.

\begin{Thm}\label{thm:main}
  Let $Y\subset\oml$ be any variety with $\cN(Y)<\infty$. Suppose that
  \begin{enumerate}
  \item $Y$ is contained in a complete-intersection defined by polynomials
    with Newton polytopes $\Delta_1,\ldots,\Delta_k$.
  \item $Y$ is set-theoretically cut out by equations with Newton
    polytope $\Delta$ containing $\Delta_\xi\^l$. To simplify the notation we write
    $\Delta_j=\Delta$ for $j>k$.
  \end{enumerate}
  Denote
  \begin{equation}
    \Gamma_j := (s+1)\Delta_\xi+\Delta_1+\cdots+\Delta_j.
  \end{equation}
  Then
  \begin{equation}
    \cN(Y) \le \sum_{j=k}^s C_{s,j} \mv(\Delta_1,\ldots,\Delta_j,\Gamma_j,\ldots,\Gamma_j).
  \end{equation}
  In particular, assuming $\Delta_j\subset\Delta$ for $j=1,\ldots,k$,
  we have
  \begin{equation}
    \cN(Y) \le E_{s,k} \mv(\Delta_1,\ldots,\Delta_k,\Delta,\ldots,\Delta),
    \qquad E_{s,k}=\sum_{j=k}^s (2s)^{s-j} C_{s,k}.
  \end{equation}
\end{Thm}

As an immediate consequence we obtain the following analog of the
Kushnirenko theorem.

\begin{Thm}\label{thm:main-kushnirenko}
  Let $Y\subset\oml$ be any variety of top dimension $s-k$ cut out by
  equations with a given Newton polytope $\Delta$ containing
  $\Delta_\xi\^l$. Suppose that $Y$ admits finitely many solutions.
  Then
  \begin{equation}
    \cN(Y) \le E_{s,k} \vol(\Delta)
  \end{equation}
\end{Thm}

More generally, for systems admitting infinitely many solutions we
have similar estimates for the degree of their reduction. Here the
degree of an irreducible variety in $\oml$ is defined to be the number
of intersections with a generic affine-linear space of complementary
dimension, and this notion is extended by linearity to arbitrary
(possible mixed-dimensional) varieties. We give an analog of the second
part of Theorem~\ref{thm:main}, but the other statements extend in
a similar manner.

\begin{Cor} \label{cor:main-deg}
  Let $Y\subset\oml$ and suppose that
  \begin{enumerate}
  \item $Y$ is contained in a complete-intersection defined by polynomials
    with Newton polytopes $\Delta_1,\ldots,\Delta_k$.
  \item $Y$ is set-theoretically cut out by equations with Newton
    polytope $\Delta$ containing $\Delta_\xi\^l$.
  \end{enumerate}
  Assume further that $\Delta_j\subset\Delta$ for $j=1,\ldots,k$. Then
  \begin{equation}
    \deg(\red Y) \le (s-k+1) E_{s,k} \mv(\Delta_1,\ldots,\Delta_k,\Delta,\ldots,\Delta).
  \end{equation}
\end{Cor}

Finally, we record a corollary of Theorem~\ref{thm:main-kushnirenko}
with a formulation more similar to that of Theorem~\ref{thm:hp-main}. In
particular, it shows that for a fixed system of differential
conditions $S$, the number of solutions within a variety $X$ defined
by equations of degree $d_X$ grows asymptotically like $d_X^n$, which
is the expected order of growth (even in a purely algebraic context).

\begin{Cor}
  Let $Y\subset\oml$ be a variety of top dimension $s-k$ cut out by
  equations of degree $d_S$ and $X\subset M$ be a variety cut out by
  equations $d_X\ge d_S$. Denote $Y=S\cap\pi^{-1}(X)$ and suppose that
  $Y$ admits finitely many solutions. Then
  \begin{equation}
    \cN(Y) \le E_{s,k} d_X^n d_S^{nl} 
  \end{equation}
\end{Cor}

\subsection{Elementary differential algebraic constructions}

For any Zariski closed set $V\subset M$ we define the $l$-th
prolongation $V\^l\subset M\^l$ to be the Zariski closure of
$\{x\^l:x\in V\}$. The prolongation of a finite union of closed sets
is clearly equal to the union of their prolongations. Moreover, by a
theorem of Kolchin \cite{kolchin}, the prolongation of an irreducible
set is itself irreducible.

The following encapsulates a key property of differentially closed
fields.

\begin{Fact}[\protect{\cite[Fact~3.7]{hp:effective}}] \label{fact:pr}
  Let $V$ denote an irreducible variety and $W\subset V\^1$ an
  irreducible variety which projects dominantly on $V$. Then for any
  non-empty Zariski open subset $U\subset W$ there exists $x\in V$
  with $x\^1\in U$.
\end{Fact}

We denote $M\^l\^1:=(M\^l)\^1$ and similarly for
$x\^l\^1$. We have a canonical variety $\Xi\subset M\^l\^1$
defined to be the Zariski closure of $\{x\^l\^1:x\in M\}$. It is given by
the equations $\xi\^k\^0=\xi\^{k-1}\^1$ for $k=1,\ldots,l$.

\begin{Prop}\label{prop:Xi}
  The points $y\in M\^l$ that satisfy $y\^1\in\Xi$ are precisely
  the points of the form $x\^l$ for some $x\in M$.
\end{Prop}
\begin{proof}
  Let $y=(x_0,x_1,\ldots,x_l)$ with $y\^1=(y,Dy)\in\Xi$. It follows
  that $x_k=Dx_{k-1}$ for $k=1,\ldots,l$, so $y=(x_0)\^l$ as claimed.
  The other direction is obvious.
\end{proof}

\begin{Lem} \label{lem:xi-projection}
  Let $Y\subset\oml$ be a variety admitting finitely many solutions.
  Then $Y\^1\cap\Xi$ does not project dominantly on any
  positive-dimensional component of $Y$.

  More generally, the same holds if $\red Y$ does not contain any
  component of $Y$.
\end{Lem}
\begin{proof}
  It is enough to check the case of $Y$ irreducible (and positive
  dimensional). Let $U=\oml\setminus\red Y$. By assumption, $U\cap Y$
  is open dense in $Y$.

  Assume toward contradiction that $Y\^1\cap\Xi$ projects dominantly
  on $Y$. Then some irreducible component $W$ of this intersection
  projects dominantly on $Y$ as well. Thus we can apply
  Fact~\ref{fact:pr} to the non-empty open subset $W\cap\pi^{-1}(U)$
  and deduce that there exists $y\in U$ with $y\^1\in Y\^1\cap\Xi$.
  But by Proposition~\ref{prop:Xi} such $y$ is of the form $z\^l$ for
  some $z\in M$, and hence $z$ is another solution of $Y$,
  contradicting the definition of $U$.
\end{proof}

\subsection{Overview of the proof}

\subsubsection{The reduction of dimension}
\label{sec:overview-dim-reduce}

Let $Y\subset\oml$ be a variety of positive dimension admitting
finitely many solutions, and let $\~Y:=\pi_{\oml}(Y\^1\cap\Xi)$. Then
by Lemma~\ref{lem:xi-projection} we have $\dim\~Y<\dim Y$. If $x\in M$
is any solution of $Y$, then $x\^l\in Y$ and hence $x\^l\^1\in Y\^1\cap\Xi$,
so $x\^l\in\~Y$. That is, $x$ is also a solution of $\~Y$. Since
$\~Y\subset Y$ we conclude that $\cN(Y)=\cN(\~Y)$.

Repeating this reduction $s$ times, one is eventually reduced to
counting the number of solutions of a zero-dimensional variety, which
is certainly bounded by the number of points in the variety. This is
similar to the approach employed in \cite{hp:effective}. In order to
obtain good effective estimates it is thus necessary to have an
effective description of $\~Y$ in terms of $Y$, and the key step is
obtaining an effective description of $Y\^1$.

\subsubsection{An approximation for the first prolongation}

We now work with an arbitrary ambient space $N$, and the reader should
keep in mind that eventually we will take this ambient space to be
$M\^l$ or its open dense subset $\oml$. To avoid confusion we denote
the coordinates on $N$ by $\zeta$.

If $V\subset N$ is an effectively smooth complete intersection defined
by $k$ polynomial equations $\{P_j\}$ with a given Newton polytope
$\Delta$, then one can explicitly write $2k$ equations in
$\zeta,\zeta\^1$ for $V\^1$, with (essentially) the same Newton
polytope in $\zeta$ and linear in $\zeta\^1$. However, this system of
equations degenerates if $V$ is a non-smooth intersection.

Our proof is based on the following idea. We embed $V$ as the zero
fiber $V=X_0$ of a flat family $X$ whose generic fiber is a complete
intersection. Making a small perturbation we may also assume that the
generic fiber is smooth. We write the system of $2k$ equations as
above (now depending on an extra deformation parameter $e$), obtaining
a family $\tau(X)$. It turns out that the limit $\tau(X)_0$, while not
necessarily equal to $V\^1$, still approximates it rather well: the
two agree over an open dense set. To conclude, if $V$ is a limit of a
family of complete intersections of a given Newton polygon, then the
same is essentially true for $V\^1$ (at least in an open dense set).

\subsubsection{Conclusion}

Returning now to the notations of~\secref{sec:overview-dim-reduce}, we
show that if $Y$ was the limit of a family of complete intersections
with a given Newton polytope, then the same is true for $\~Y$ (with a
slightly larger Newton polytope). We may now repeat the construction
$s$ times and eventually obtain the limit of a zero-dimensional
variety. Of course, the number of points of such a limit is bounded by
the number of points of the generic fiber, which may now be estimated
with the help of the BKK theorem.

\subsection{Acknowledgements}

I would like to express my gratitude to Ehud Hrushovski for suggesting
the problem studied in this paper to me and for invaluable discussions.
I also wish to thank Omar Leon Sanchez for his useful suggestions during
the preparation of this manuscript.

\section{Smooth approximations and flat families}

In this section we denote the ambient space by $N$ and its dimension
by $n$. The reader should keep in mind that eventually we will take
this ambient space to be $M\^l$ or its open dense subset $\oml$. To
avoid confusion we denote the coordinates on $N$ by $\zeta$.

\subsection{The $\tau$-variety associated to a flat family}

Let $V\subset N$ be given by
\begin{equation} \label{V-def}
  V = \{ P_1=\cdots=P_k=0 \}, \qquad P_j\in K[\zeta].
\end{equation}
We define $\tau(V)\subset N\^1$ to be
\begin{equation}
  \tau(V) := \{ (\zeta,\zeta\^1) : P_j(\zeta)=0, (\d P_j)_\zeta(\zeta\^1)+P_j^D(\zeta)=0, j=1,\ldots,k\}
\end{equation}
where $P^D$ is the polynomial obtained by applying $D$ to each
coefficient of $P$. It follows from the chain rule that
$V\^1\subset\tau(V)$. Moreover, if $V$ is effectively smooth then
equality holds \cite[Remark 3.5.(3)]{hp:effective}, and in particular
$\tau(V)$ has pure codimension $2k$ in $N\^1$.

Recall that a variety $X\subset N\times\A_K$ is flat over $\A_K$ if
and only if every component of $X$ projects dominantly on $\A_K$. For
general $X$, we denote by $\cF(X)$ the flat family obtained by
removing any component which projects to a point in $\A_K$. For any
$\e\in\A_K$ we denote by $X_\e$ the $\e$-fiber of $X$.

Consider now a variety $X\subset N\times\A_K$ given by
\begin{equation} \label{eq:X-def}
  X = \cF\big[\{ P_1=\cdots=P_k=0 \}\big], \qquad P_j\in K[\zeta,e]
\end{equation}
where $e$ denotes a coordinate on $\A_K$. We say that $X$ is a
\emph{generic complete intersection} if the generic fiber $X_\e$ has
pure codimension $k$. We define $\tau(X)\subset N\^1\times\A_K$ to be
\begin{equation} \label{eq:tau-def}
  \tau(X) := \cF\big[\{ (\zeta,\zeta\^1) :
  P_j(\zeta)=0, (\d P_j)_\zeta(\zeta\^1)+P_j^D(\zeta)=0, j=1,\ldots,k\}\big].
\end{equation}
For generic $\e\in K_0$, $\tau(X)_\e$ is just $\tau(X_\e)$. The
following subsection establishes a precise sense in which $\tau(X)_0$,
obtained as the limit of these generic fibers, approximates
$(X_0)\^1$.

\subsection{Approximation of the first prolongation}

In this section it will be convenient for us to assume that the field
$K$ has an analytic realization. We therefore assume that the field of
constants $K_0$ is a subfield of $\C$, and that any all functions
involved in the definition of any of the varieties we consider have
been embedded in the field of meromorphic functions on the disc $\D$
(this is always possible by a result of Seidenberg
\cite{seidenberg:diff-algebra-I,seidenberg:diff-algebra-II}). Thus we
may consider $K$-varieties as analytic sets.

We begin with a simple lemma.
\begin{Lem} \label{lem:seq-select}
  Let $X\subset N\times\A_K$ be a flat family, and $x\in X_0$. Then
  for any sequence $\e_i\in K_0$ with $\e_i\to0$, there exists a
  sequence of $K$-points $x_i\in X_{\e_i}$, holomorphic in a common disc
  and converging uniformly to $x$.
\end{Lem}
\begin{proof}
  Intersecting with generic linear functionals vanishing at $x$ one
  can reduce the problem to the case that $X$ is a curve. Moreover,
  changing coordinate $e\to e^\nu$ we may assume that the curve is
  irreducible, smooth and transversal to $e$ at $x$ in the $K$-sense,
  i.e. for generic $t$. We restrict the disc $\D$ to make this true
  for every $t$.

  Consider the intersection $X\cap\{e=\e_i\}$. This is a zero
  dimensional set, and for sufficiently small $\e_i$ contains
  precisely one solution $x_i$ near $x$. All $x_i$ are $K$-points and
  moreover, are in fact defined over the field of definition of $X$.
  Therefore they may be viewed as (a-priori, ramified) holomorphic
  functions on $\D$. But in fact $x_i(t)$ is well-defined for
  $t\in\D$, and thus no ramification can occur and the functions
  $x_i(t)$ are holomorphic in $\D$. Finally, $x_i$ converges pointwise
  to $x$ by definition, and it follows by standard arguments that
  convergence is uniform (perhaps on a smaller disc).  
\end{proof}

We remark that since the Zariski topology is coarser than the analytic
$\C$-topology, the converse also holds: if $x_i\in X_{\e_i}$ is a
sequence of $K$ points holomorphic on a disc $\D$ and converging
uniformly to $x\in K$, then any Zariski closed set containing $x_i$ for
every $i$ must contain $x$.

\begin{Prop} \label{prop:pr-in-tau}
  Let $X$ be a generic complete intersection as in~\eqref{eq:X-def}.
  Then we have $(X_0)\^1\subset\tau(X)_0$.
\end{Prop}
\begin{proof}
  Since $\tau(X)_0$ is closed by definition, it is enough to prove
  that for every $x\in X_0$, we have $x\^1\in\tau(X)_0$. Let $x\in X_0$.
  Since the family $X$ is flat, we can choose by
  Lemma~\ref{lem:seq-select} a sequence $x_i\in X_{\e_i}$ with
  $\e_i\in K_0$ and $\e_i\to0$, converging uniformly to $x$. Then
  $(x_i)\^1$ converges uniformly to $x\^l$ (in an appropriate disc).
  Since $(X_{\e_i})\^1\subset\tau(X_{\e_i})$ and
  $\tau(X_{\e_i})=\tau(X)_{\e_i}$ by the remark
  following~\eqref{eq:tau-def}, we have $(x_i)\^1\in\tau(X)_{\e_i}$.
  Since $\tau(X)$ is closed we conclude that $x\^l\in\tau(X)_0$ as
  claimed
\end{proof}

For a variety $V\subset N$ as in~\eqref{V-def}, we can describe
$\tau(V)$ at effectively smooth points in terms of complex tangent
spaces. Indeed, let $x\in V$ be an effectively smooth $K$-point. Then
$V$ is smooth as a complex variety at $x(t)$ for generic $t$ and we
have the equivalence
\begin{equation} \label{eq:tau-desc}
  (x,v)\in\tau(V) \iff \partial_t+v\cdot\partial_\zeta \in T_x V
\end{equation}
where the right hand side is understood to hold for generic $t$. In
particular, if $V$ is any variety then one can choose a system of
equations such that it becomes effectively smooth at every smooth
point. Thus, since we know that $\tau(V)$ agrees with $V\^1$ over the
effectively smooth locus, we conclude that for every smooth point
$x\in V$ we also have
\begin{equation} \label{eq:pr-desc}
  (x,v)\in V\^1 \iff \partial_t+v\cdot\partial_\zeta \in T_x V.
\end{equation}

The following proposition establishes a partial converse to
Proposition~\ref{prop:pr-in-tau}.

\begin{Prop}\label{prop:pr-eq-tau}
  Let $X$ be a generic complete intersection as in~\eqref{eq:X-def}
  such that the generic fiber $X_\e$ is effectively smooth. Then there
  is an open dense set $U\subset X_0$ such that
  \begin{equation}
    (X_0)\^1\cap\pi^{-1}(U)=\tau(X)_0\cap\pi^{-1}(U).
  \end{equation}
\end{Prop}

Before proving this proposition, we illustrate with the following
example.

\begin{Ex} \label{ex:pr-eq-tau}
  Consider $X\subset \A_K^2\times\A_K$ given by $y^2=ex$. Since the family
  is defined by an equation with constant coefficients, $(X_0)\^1$ agrees
  with the tangent bundle of $X_0$. For generic $\e$, the fiber $X_\e$ is
  an effectively smooth parabola, and for $\e$ constant $\tau(X)_\e$ agrees
  with the tangent bundle of this parabola.

  The fiber $\tau(X)_0$ is the flat limit of $\tau(X)_e$. Since
  $\partial_y$ is tangent to $X_\e$ at $(x,y)=(0,0)$ for any $\e$, the
  fiber of $\tau(X)_0$ over this point contains $\partial_y$. Thus
  $\tau(X)_0$ is strictly larger than the tangent bundle of $X_0$.
  However, over any point of $U:=X_0\setminus\{(0,0)\}$, can easily
  check that the limit of $X_\e$ agrees with the tangent bundle of
  $X_0$.

  \begin{figure}
    \centering
    \def\svgwidth{\textwidth}
    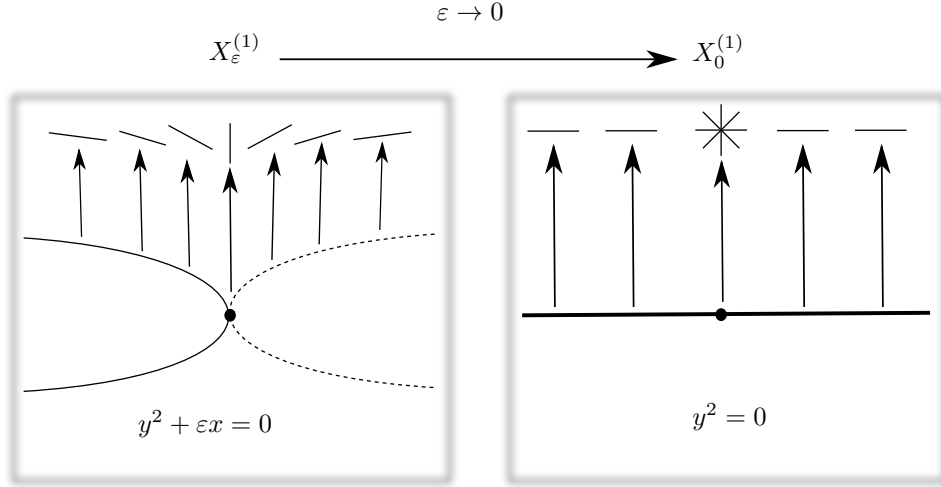
    \caption{Degeneration of $\tau(X)$ in Example~\ref{ex:pr-eq-tau}}
  \end{figure}
\end{Ex}

\begin{proof}[Proof of Proposition~\ref{prop:pr-eq-tau}]
  To avoid confusion, in this proof we make the convention that $p$
  denotes $\C$-points, whereas $x$ denotes $K$-points.
  
  Since the map $e$ is flat on $X$ (over $K$), it follows that for
  generic $t$ there are no components of $X$ contained in the $e=0$
  fiber. After restricting to a disc where this holds, the map
  $e:X\to\C$ can be viewed as a flat map over $\C$. By a theorem of
  Hironaka \cite{hironaka}, there exists an analytic stratification
  $\{Z_\alpha\}$ of $X_0$ with the following property: \emph{for any
    sequence of points $p_i\in X$ converging to $p\in Z_\alpha$, if
    $X_{e(p_i)}$ is smooth at $p_i$ and $T_{p_i}X_{e(p_i)}$ converges
    to a limit $T$, then $T_p Z_\alpha\subset T$}. This is somewhat
  weaker form of Thom's $A_e$-condition, which will suffice for our
  purposes. We fix such a stratification, and let $Z_0$ denote the
  (union of) top-dimensional stratas.

  The following observation is the key geometric idea for the proof.
  Let $p_i\in X$ be any sequence of points converging to
  $p\in Z_0$ with $X_{e(p_i)}$ smooth at $p_i$ and
  $v_i\in T_{p_i}X_{e(p_i)}$ a sequence of tangent vectors converging
  to a vector $v$. Then $v\in T_pX_0$. Indeed, one can always
  pass to a subsequence such that $T_{p_i}X_{e(p_i)}$ converges to
  some limit $T$, and hence $T_p Z_0\subset T$. But since the
  dimensions of these sets agree, in fact $T_pX_0=T_pZ_0=T$, and in
  particular $v\in T$ implies $v\in T_pX_0$.
  
  We know by Proposition~\ref{prop:pr-in-tau} that
  $(X_0)\^1\subset\tau(X)_0$. The set of points in $X_0$ such that the
  $\pi$-fibers of these two sets are equal is $K$-constructible. Thus,
  if it does not contain an open dense set it must be contained in a
  closed set of strictly smaller dimension. Therefore it will suffice
  to establish the claim over any \emph{analytic} open dense set $U$.
  We will establish it with $U=Z_0$. Thus, let $(x,v)\in\tau(X)_0$
  with $x\in Z_0$, and we will prove that $(x,v)\in(X_0)\^1$.

  Since $\tau(X)$ is flat by definition, we may by
  Lemma~\ref{lem:seq-select} choose a sequence $\e_i\in K_0$ with
  $\e_i\to0$ and a sequence $(x_i,v_i)\in \tau(X)_{\e_i}$ such that
  $(x_i,v_i)$ are defined in a common disc $\D$ and converge uniformly
  to $(x,v)$. By assumption, we may take the fibers $X_{\e_i}$ to be
  effectively smooth. Also, by the remark following~\eqref{eq:tau-def}
  we may assume that $\tau(X)_{\e_i}=\tau(X_{\e_i})$. We conclude
  from~\eqref{eq:tau-desc} that
  \begin{equation}
    \partial_t+v_i\cdot\partial_\zeta\in T_{x_i} X_{\e_i}
  \end{equation}
  for generic $t$.

  For $t$ outside a countable set, each fiber $X_{\e_i}$ is
  effectively smooth (over $\C$) at $x_i(t)$. By the key geometric
  observation above it now follows that
  \begin{equation}
    \partial_t+v\cdot\partial_\zeta\in T_x X_0
  \end{equation}
  for generic $t$. Thus~\eqref{eq:pr-desc} gives $(x,v)\in(X_0)\^1$ as
  claimed.
\end{proof}

\begin{Cor}\label{cor:tau-components}
  Let $X$ be a generic complete intersection as in~\eqref{eq:X-def}
  and suppose that the generic fiber $X_\e$ is effectively smooth. Then
  every component of $(X_0)\^1$ is a component of $\tau(X)_0$.
\end{Cor}
\begin{proof}
  Indeed, if $Z_1,\ldots,Z_r$ denote the components of $X_0$ then
  they each have codimension $k$, and $Z_i\^1$ are the components of
  $(X_0)\^1$. In particular, $\pi^{-1}(U)$ is dense in each $Z_i\^1$
  (for $U$ given in Proposition~\ref{prop:pr-eq-tau}), and since
  $\tau(X)_0$ agrees with $(X_0)\^1$ over $\pi^{-1}(U)$ it follows that each
  $Z_i\^1$ is also a component of $\tau(X)_0$.
\end{proof}

\section{Constructions with flat families}

Once again, in this section we denote the ambient space by $N$ and its
dimension by $n$. The reader should keep in mind that eventually we
will take this ambient space to be $M\^l$ or its open dense subset
$\oml$. To avoid confusion we denote the coordinates on $N$ by
$\zeta$.

When speaking about the Newton polytope of a polynomial in $K[\zeta,e]$
we mean the Newton polytope in the $\zeta$ variables obtained for a
generic value of $e$.

\subsection{General lemmas on perturbations and intersections in flat families}

The following proposition shows that the $0$-fiber of a flat family is
not changed if one makes a sufficiently small perturbation of the
family, i.e. a perturbation of sufficiently high order in $e$.

\begin{Prop}\label{prop:def-pert}
  Let $X\subset N\times\A^2_K$ be a variety, and suppose that every
  component of $V$ projects dominantly on $\A^2_K$. Then the flat
  limit at the origin of $X_{s,t}$ (as a variety) along
  $\Gamma=\{t=0\}\subset\A^2_K$ is the same as the flat limit at the
  origin along any smooth curve $\Gamma'\subset\A^2_K$ sufficiently
  tangent to $\Gamma$.
\end{Prop}
\begin{proof}
  We may assume that $X$ is irreducible. Denote the projection to
  $\A_K^2$ by $\eta$. Recall that $N$ is an open dense subset of some
  projective space $\C P^S$. Let $\bar X$ denote the closure of $X$ in
  $\C P^S\times \A^2_K$, which clearly projects dominantly to $\A^2_K$
  as well.

  Consider the map $\Phi$ taking a pair $(s,t)$ to the fiber
  $X_{s,t}:=\bar X\cap\eta^{-1}(s,t)$ with its cycle structure in the
  Chow variety parametrizing cycles of degree $\deg X$ in $\C P^S$.
  This is a rational map defined over the (open dense) locus
  $U=\A^2_K\setminus\Sigma$ such that the intersection above is
  proper. We have $\codim\Sigma\ge2$ (otherwise
  $\dim\eta^{-1}(\Sigma)=\dim X$, which is ruled out by the
  hypotheses). Thus $\Phi$ is defined on both $\Gamma$ and $\Gamma'$
  except for possibly finitely many exceptions where they intersect
  $\Sigma$.

  In projective coordinates on the Chow variety, we have
  \begin{equation}
    \Phi(s,0) = s^\nu [\~X] + O(s^{\nu+1}) 
  \end{equation}
  where $[\~X]$ denotes the Chow form of the flat limit of $X_{s,t}$ along $\Gamma$
  (with its cycle structure). Thus
  \begin{equation}
    \Phi(s,t) = s^\nu[\~X] + O(s^{\nu+1})+O(t).
  \end{equation}
  We see that if $t=O(s^{\nu+1})$ along $\Gamma'$ then the limit of
  $X_{s,t}$ along $\Gamma'$ is also equal to $[\~X]$ as a cycle. In
  particular, these cycles intersect $N$ in the same set, proving the
  claim.
\end{proof}

\begin{Rem}
  Let $U$ denote a Zariski open subset in the space of $n$-tuples of
  polynomials with a given Newton polytope.
  Proposition~\ref{prop:def-pert} implies, in particular, that one may
  deform a given complete intersection family~\eqref{eq:X-def} to make
  $P_i(\zeta,e)\in U$ for generic values of $e$, without changing the
  fiber $X_0$. Indeed, consider $\~P_i=P_i+sQ_i$ where $Q_i$ denotes
  some tuple of polynomials from the space. Then taking $s=e^\nu$ for
  sufficiently large $\nu$ ensures that the families defined by
  $\{P_i\}$ and $\{\~P_i\}$ have the same $0$-fiber, whereas an
  appropriate (generic) choice of $Q_i$ ensures that $\{\~P_i\}\in U$
  for generic $e$.

  For instance, in conjunction with the Bertini theorem, this implies
  that one can make the generic fiber effectively smooth.
\end{Rem}

If $X$ is a flat family and $P\in K[\zeta]$ is such that $\{P=0\}$
intersects $X_0$ properly, then the family
$Y=\cF\big[X\cap\{P=0\}\big]$ satisfies $Y_0=X_0\cap\{P=0\}$. However,
if the intersection is not proper, one cannot in general predict the
structure of $Y_0$. The following proposition shows that under a
technical modification, one can guarantee that $Y_0$ is given by the
intersection between $\{P=0\}$ and those components of $X_0$ that meet
it properly.

\begin{Prop}\label{prop:def-intersect}
  Let $X\subset N\times\A_K$ be a generic complete intersection
  and $P\in K[\zeta]$. Define $Y,\~Y$ by
  \begin{align}
    Y &= \cF\big[ X\cap \{ P=e^{1/\nu} \}\big] \\
    \~Y &= \cF\big[ (X_0\times\A_K) \cap \{ P=e^{1/\nu} \}\big]
  \end{align}
  where $\nu$ is a sufficiently large natural number.

  Then $\~Y_0=Y_0$. In particular, if $X_0=C\cup C'$ where $C'$ denotes the
  union of components of $X_0$ where $P$ vanishes identically and $C$ the
  rest, then $Y_0= C\cap\{P=0\}$.
\end{Prop}
\begin{proof}
  The first part of the statement is simply
  Proposition~\ref{prop:def-pert} applied to the family $(X\times\A_K^s)\cap\{P=s\}$
  (with $e$ in place of $t$). For the second part it suffices to
  compute $\~Y_0$ by noting that $\{P=e^{1/\nu}\}$ intersects $C$
  properly and does not intersect $C'$ at all (in the affine space).
\end{proof}

We remark that the use of the ramified factor $e^{1/\nu}$ is merely
a notational convenience, to avoid reparametrizing the original
family $X$. One could of course obtain an honest polynomial family
by passing to a new deformation parameter.

\begin{Ex}\label{ex:def-intersect}
  Consider $X\subset\C^2\times\C$ given by $xy=e$, and $P=x$. Then
  $X_0$ has two components, the $x$ and the $y$ axes. The
  former intersects $\{P=0\}$ properly at the origin, whereas
  the latter intersects $\{P=0\}$ non-properly.
  
  If we consider the family $Y=\cF(X\cap\{P=0\})$ we obtain
  $Y_0=\emptyset$. If we consider $Y=\cF(X\cap\{P=e\})$ we obtain
  $Y_0=\{(0,1)\}$. Finally, if we consider $Y=\cF(X\cap\{P=e^\alpha\}$
  where $0<\alpha<1$, then we obtain $Y=\{(0,0)\}$, which is the
  intersection between the $x$-axis and $\{P=0\}$ as predicted by
  Proposition~\ref{prop:def-intersect}.

  \begin{figure}
    \centering
    \includegraphics[width=8cm]{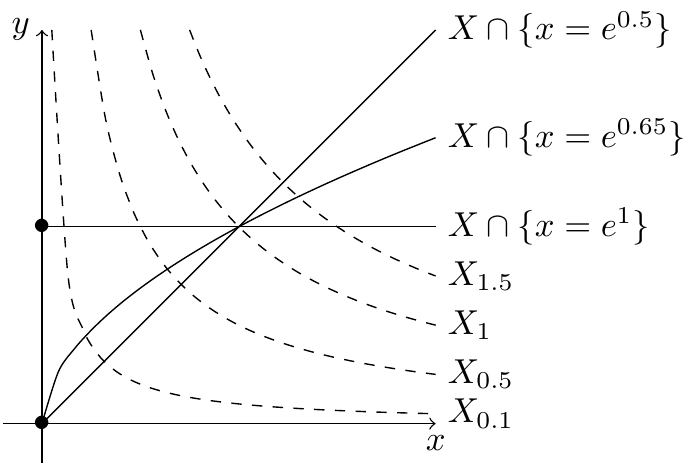}
    \caption{The different families in Example~\ref{ex:def-intersect}}
  \end{figure}
\end{Ex}

\subsection{Projections and reduction of dimension}

We begin with a simple lemma on linear elimination of
variables. Let $P\in K[\zeta,\zeta\^1]$ be a polynomial with
$\deg_{\zeta\^1} P\le1$. The homogenization $\~P$ of $P$ is the
polynomial in one extra variable, $\zeta_0\^1,\ldots,\zeta_n\^1$ where
the free term of $P$ is multiplied by $\zeta_0\^1$. The following
lemma is immediate.

\begin{Lem} \label{lem:lin-elimination}
  Let
  \begin{equation}
    P_1,\ldots,P_{n+1}\in K[\zeta,\zeta\^1] \qquad \Delta(P_j)\subset \Delta_j\times\Delta_{\zeta\^1}
  \end{equation}
  and let $\~P_1,\ldots,\~P_{n+1}$ denote their
  homogenization, and $W\subset N\times\P^n_K$ the variety they
  define.

  Let $R$ denote the determinant of the $(n+1)\times(n+1)$ matrix
  whose $j$-th row are given by the $n+1$ coefficients of
  $\~P_j$. Then $\Delta(R)\subset\sum_{j=1}^{n+1} \Delta_j$ and $\pi(W)=\{R=0\}$.
\end{Lem}

\begin{Lem}\label{lem:projection}
  Let
  \begin{equation}
    P_1,\ldots,P_k\in K[\zeta,e]
  \end{equation}
  and
  \begin{equation}
    Q_1,\ldots,Q_l\in K[\zeta,\zeta\^1,e] \qquad \Delta(Q_j)\subset \Delta_j \times\Delta_{\zeta\^1}
  \end{equation}
  define a generic complete intersection $X\subset N\^1\times\A_K$,
  \begin{equation}
    X = \cF\big[\{ P_1=\cdots=P_k=Q_1=\cdots=Q_l=0 \}\big].
  \end{equation}
  Denote by $W$ the union of the components $C$ of $X_0$ such that
  $\codim\pi(C)>k$.
  
  Then there exists $R\in K[\zeta,e]$ and a generic complete
  intersection $Y$,
  \begin{equation}\label{eq:proj-family}
    Y=\cF\big[\{P_1=\cdots=P_k=R=0\}\big], \qquad \Delta(R)\subset \sum_{j=1}^l \Delta_j
  \end{equation}
  such that $\pi(W)\subset Y_0$.
\end{Lem}
\begin{proof}
  We prove the claim by reverse induction on $l$, starting with the
  case $l\ge n+1$. In this case $W=X_0$. We replace
  $Q_1,\ldots,Q_{n+1}$ by their small generic perturbation (keeping
  the same Newton polytope). By Proposition~\ref{prop:def-pert} this
  does not change $X_0$. We can thus assume without loss of generality
  that the generic fiber of $V\subset N\times\P^n_K\times\A_K$ cut out
  by $P_1,\ldots,P_k$ and the homogenizations of $Q_1,\ldots,Q_{n+1}$
  has codimension $n+k+1$. Applying Lemma~\ref{lem:lin-elimination} to
  $Q_1,\ldots,Q_{n+1}$ we thus obtain $R$ giving a generic complete
  intersection~\eqref{eq:proj-family}. By construction the generic
  fiber $Y_\e$ contains $\pi(X_\e)$, and it follows that $Y_0$
  contains $\pi(X_0)$ as claimed.

  We now consider the case $l\le n$. Then for any component
  $C\subset W$, the generic fiber of $\pi:C\to\pi(C)$ has dimension at
  least $n+1-l\ge1$. Let $\ell$ be a generic affine-linear form in the
  $\zeta\^1$ variables. Then $\{\ell=0\}$ meets $X_0$ properly, and taking
  $Q_{l+1}$ to be $\ell$ we obtain a new family $X'$ such that
  \begin{equation}
    C' := C\cap\{\ell=0\} \subset X'_0.
  \end{equation}
  Moreover,
  \begin{equation}
    \clo[\pi(C)]=\clo[\pi(C')]
  \end{equation}
  since $\{\ell=0\}$ meets the generic fiber of $\pi:C\to\pi(C)$.
  Since this is true for any component $C\subset W$ (with sufficiently
  generic $\ell$), the claim follows by induction with the new family
  $X'$ and $\Delta_{l+1}=\{0\}$.
\end{proof}

The following lemma provides the main inductive step in our estimates.

\begin{Lem} \label{lem:dim-reduce}
  Let $X\subset N\times\A_K$ given by
  \begin{equation}
    X = \cF\big[\{ P_1=\cdots=P_k=0 \}\big], \qquad P_j\in K[\zeta,e],\ \Delta(P_j)\subset\Delta_j
  \end{equation}
  be a generic complete intersection. Let $S\subset N\^1$ be a variety
  defined by equations which are affine-linear in $\zeta,\zeta\^1$.
  Suppose that $(X_0)\^1\cap S$ does not project dominantly on any
  component of $X_0$.

  Then there exist
  \begin{equation}
    \~P_1,\ldots,\~P_{k+1} \in K[\zeta,e]
  \end{equation}
  with $\Delta(\~P_j)\subset \Delta_j$ for $j=1,\ldots,k$ and
  \begin{equation} \label{eq:gamma-def}
    \Delta(\~P_{k+1}) \subset (n-k+1)\Gamma, \qquad 
    \Gamma=(n+1)\Delta_\zeta + \sum_{i=1}^k \Delta_j
  \end{equation}
  such that
  \begin{equation}
    Y = \cF\big[\{ \~P_1=\cdots= \~P_{k+1}=0 \}\big]
  \end{equation}
  is a generic complete intersection and
  $\pi((X_0)\^1\cap S)\subset Y_0\subset X_0$.
\end{Lem}
\begin{proof}
  By Proposition~\ref{prop:def-pert} any sufficiently small
  perturbation of $P_1,\ldots,P_k$ does not change $X_0$. We fix such
  a perturbation making the generic fiber of the family effectively
  smooth (without changing the Newton polytopes). Without loss of
  generality we may assume that the family $P_1,\ldots,P_k$ is already
  in this form.

  We let $Z^0$ denote the family $\tau(X)$. Recall that it is a
  generic complete intersection defined by the vanishing of
  $P_1,\ldots,P_k$ and another set of $k$ equations $Q_1,\ldots,Q_k$
  with
  $\Delta(Q_j)\subset(\Delta+\Delta_\zeta)\times\Delta_{\zeta\^1}$. By
  Corollary~\ref{cor:tau-components} every component of $(X_0)\^1$ is
  also a component of $Z^0_0$. Moreover, any extra components do not
  project dominantly on a component of $X_0$ by
  Proposition~\ref{prop:pr-eq-tau}.

  For $j=1,\ldots,n-k+1$ we define $Z^j$ to be the family obtained
  from $Z^{j-1}$ by intersecting with a generic linear combination
  $\~Q^j$ of the given equations of $S$, using
  Proposition~\ref{prop:def-intersect}. We define $R^j,Y^j$ by
  applying Lemma~\ref{lem:projection} to $Z^j$. We have
  $\Delta(R_j)\subset\Gamma$. Finally, we take
  $\~P_{k+1}=R_1\cdots R_{n-k+1}$ which gives
  $Y=Y^1\cup\cdots\cup Y^{n-k+1}$.

  Let $C$ be a component of $(X_0)\^1$. Then it is also a component of
  $Z^0_0$. Since $C$ projects dominantly on a component of $X_0$, it
  cannot be contained in $S$. Hence by genericity $C':=C\cap\{\~Q^1=0\}$ is strictly
  contained in $C$, and by Proposition~\ref{prop:def-intersect} we
  have $C'\subset Z^1_0$. Moreover, clearly $C'\cap S=C\cap S$. Thus it will suffice
  to prove that
  \begin{equation}
    \pi(Z_0^1 \cap S) \subset Y_0.
  \end{equation}

  Now let $C$ be any component of $Z_0^1$. If $\codim(\pi(C))>k$ then
  \begin{equation}\label{eq:C-proj-ok}
    \pi(C\cap S)\subset \pi(C)\subset Y^1_0 \subset Y_0
  \end{equation}
  where the middle inclusion follows from Lemma~\ref{lem:projection}.
  Otherwise, $C$ projects dominantly on a component of $X_0$, so it
  again cannot be contained in $S$. Again by genericity
  $C':=C\cap\{\~Q^2=0\}$ is strictly contained in $C$, and by
  Proposition~\ref{prop:def-intersect} we have $C'\subset Z^2_0$.
  Moreover, clearly $C\cap S=C'\cap S$. Thus it will suffice to
  prove that
  \begin{equation}
    \pi(Z_0^2 \cap S) \subset Y_0.
  \end{equation}

  We repeat this process until $Z_0^{n-k+1}$. At this point every component
  has dimension at most $n-k-1$ even before the $\pi$ projection, and thus
  project to a set of codimension greater than $k$. As in~\eqref{eq:C-proj-ok}
  it then follows that
  \begin{equation}
    \pi(C\cap S)\subset \pi(C)\subset Y^{n-k+1}_0 \subset Y_0
  \end{equation}
  for every component $C$ of $Z_0^{n-k+1}$, thus concluding the proof.
\end{proof}

\begin{Rem} \label{rem:refined-gamma} If the polytopes
  $\Delta_1,\ldots,\Delta_k$ are co-ideal polytopes (see
  Remark~\ref{rem:co-ideal-poly}) then they are stable under
  differentiations. In this case in the definition of $\Gamma$
  in~\eqref{eq:gamma-def} one can replace $(n+1)\Delta_\zeta$ by
  $(n-k+1)\Delta_\zeta$. For such polytopes, all of our main results
  remain valid with this refined definition of $\Gamma$.
\end{Rem}

\section{Proofs of the main estimates}

In this section we prove our main results on the number of isolated
solutions of a differential system of equations and the degree of its
reduction. We begin with a lemma demonstrating the application of
Lemma~\ref{lem:dim-reduce} to this context.

\begin{Lem} \label{lem:dim-reduce-xi}
  Let $X\subset\oml\times\A_K$ be a generic complete intersection
  family, and suppose that $(X_0)\^1\cap \Xi$ does not project
  dominantly on any component of $X_0$. Then the family $Y$ defined by
  applying Lemma~\ref{lem:dim-reduce} to $X$ with $M=\oml$ and $S=\Xi$
  satisfies $X_0\cap\J^l(M)=Y_0\cap\J^l(M)$.
\end{Lem}
\begin{proof}
  Since $Y_0\subset X_0$, one inclusion is obvious. In the other
  direction, let $x\^l\in X_0$. Then $x\^l\^1\in(X_0)\^1\cap \Xi$ so
  by Lemma~\ref{lem:dim-reduce} $x\^l=\pi(x\^l\^1)\in Y_0$ as well.
\end{proof}

We now prove our result for the case that $Y$ is a fiber $X_0$ of a
generic complete intersection family $X$. In particular the following
proposition implies (and generalizes)
Theorem~\ref{thm:comp-intersections}.

\begin{Prop} \label{prop:flat-limit-bound}
  Let $X\subset\oml\times\A_K$ be a generic complete intersection
  family defined by polynomials with Newton polytopes
  $\Delta_1,\ldots,\Delta_k$, and suppose that $X_0$ admits finitely
  many solutions. Denote
  \begin{equation}
    \Gamma=s\Delta_\xi\^l+\Delta_1+\cdots+\Delta_k.
  \end{equation}
  Then
  \begin{equation}
    \cN(X_0) \le C_{s,k} \mv(\Delta_1,\ldots,\Delta_k,\Gamma,\ldots,\Gamma),
    \qquad C_{s,k} = s! (\delta+2)^{\delta(\delta+1)/2}
  \end{equation}
  where $\delta:=s-k$.
\end{Prop}
\begin{proof}
  Let $X(0)=X$. As long as $X(j)_0$ has positive dimension, we
  define $X(j+1)$ to be the family obtained by applying
  Lemma~\ref{lem:dim-reduce} to $X(j)$ with the ambient space $\oml$
  and $S=\Xi$. The lemma is applicable by application of
  Lemma~\ref{lem:xi-projection}, since $X(j)_0$ admits finitely many
  solutions. We denote the extra equation obtained in this
  process\footnote{In fact each application of
    Lemma~\ref{lem:dim-reduce} perturbs all the equations defining
    $X(j)$; but the Newton polytopes remain unchanged.} by $R_{j+1}$. By
  Lemma~\ref{lem:dim-reduce-xi} we have $\cN(X(j+1)_0)=\cN(X(j)_0)$.

  The Newton polytope of $R_j$ is contained in $(\delta+2)^j \Gamma$, as one
  can see by the following simple induction:
  \begin{equation}
    \Delta(R_j)\subset (\delta+1)\big(\sum_{i=0}^{j-1} (\delta+2)^i \Gamma\big) \subset 
    (\delta+1)\frac{ (\delta+2)^j-1}{\delta+1} \Gamma
    \subset (\delta+2)^j \Gamma
  \end{equation}

  Eventually we obtain the zero-dimensional variety $X(\delta)$. From the
  above we conclude that $\cN(X_0)=\cN(X(\delta)_0)$, which is certainly
  bounded by the number of points in $X(\delta)_0$. Since the number of
  points of a flat limit is certainly bounded by the number of points
  of the generic fiber, we have by the BKK theorem
  \begin{align}
    \cN(X_0) &\le s! \mv(\Delta_1,\ldots,\Delta_k,(\delta+2)^1\Gamma,\ldots,(\delta+2)^{\delta}\Gamma) \\
    &= (\delta+2)^{\delta(\delta+1)/2} s! \mv(\Delta_1,\ldots,\Delta_k,\Gamma,\ldots,\Gamma)
  \end{align}
  as claimed.
\end{proof}

We now prove our result for arbitrary varieties $Y\subset\oml$.

\begin{proof}[Proof of Theorem~\ref{thm:main}]
  Let $X(k)\subset\oml\times\A_K$ be the (constant) family cut out by
  the given equations with Newton polytopes
  $\Delta_1,\ldots,\Delta_k$.

  For $j\ge k$, let $P_j$ denote a generic linear combination of the
  given equations with Newton polytope $\Delta$ defining $Y$. Define
  $X(j+1)$ to be the family obtained by application of
  Proposition~\ref{prop:def-intersect} to $X(j)$ and $P_j$. Write
  $X(j)_0=A(j)\cup B(j)$, where $A(j)$ denotes the union of the
  components of $X(j)_0$ that are contained in $Y$, and $B(j)$ the
  rest.

  The number of solutions of $Y$ which are contained in
  $C(j)=A(j)\setminus B(j)$ is bounded by
  \begin{equation}
    \cN(C(j)) \le C_{s,k} \mv(\Delta_1,\ldots,\Delta_j,\Gamma_j,\ldots,\Gamma_j).
  \end{equation}
  Indeed, $C(j)$ is the flat limit of the family $X(j)$ in the ambient
  space $\oml\setminus B(j)$, and the bound thus follows from
  Proposition~\ref{prop:flat-limit-bound}. The claim of the theorem
  will follow once we show that $Y=C(k)\cup\cdots\cup C(s)$.

  Let $x\in Y$, and we will show that it belongs to some $C(j)$.
  Certainly $x\in X(k)_0$. If $X\not\in B(k)$, we are done. Otherwise
  $x$ is contained in some component $G\subset X(k)_0$ with
  $G\not\subset Y$, and we may assume by genericity that $P_k$ does
  not vanish identically on $G$. Then according to
  Proposition~\ref{prop:def-intersect}, $X(k+1)_0$ contains
  $G\cap\{P_k=0\}$, and in particular $x\in X(k+1)_0$.

  We continue in the same manner. The process must stop at $j=s$
  (if not before), because at this point $X(j)_0$ consists of isolated
  points, so $x\in X(j)_0$ implies $x\in C(j)$.
\end{proof}

Finally we prove the more general Corollary~\ref{cor:main-deg}.

\begin{proof}[Proof of Corollary~\ref{cor:main-deg}]
  We indicate the minor changes required in the proofs of
  Proposition~\ref{prop:flat-limit-bound} and Theorem~\ref{thm:main}.

  The proof of Proposition~\ref{prop:flat-limit-bound} carries out
  verbatim as long as $X(j)_0$ does not have a component which is a
  component of $\red Y$. Let $S(j)\subset\red Y$ denote the components
  of codimension $j$. Then these are also components of $X(j)_0$, and
  their degrees can be bounded by the BKK theorem. We then define
  $X(j+1)$ in the same way, but restricting the ambient space to
  $\oml\setminus (S(1)\cup\cdots\cup S(j))$. This insures that
  Lemma~\ref{lem:dim-reduce} is again applicable. It remain only to
  note that the degree of the remaining components of $\red Y$ in the
  new ambient space is the same as the degree in the original ambient
  space. The proof is thus concluded by induction. The resulting bound
  has an extra multiplicative factor of $s-k+1$ corresponding to the
  fact that we separately bound the degree in each
  dimension\footnote{one could obviously obtain a sharper estimate
    taking into account the different BKK estimate for each
    dimension.}.

  We can now carry out the proof of Theorem~\ref{thm:main} in the same
  way, noting that since $Y=C(k)\cup\cdots\cup C(s)$, any component of
  $\red Y$ must be a component of $\red{C(j)}$ for some
  $j=k,\ldots,s$.
\end{proof}

\section{Diophantine applications}
\label{sec:applications}

In the papers \cite{hp:effective,fs:j-func} the effective estimate of
Theorem~\ref{thm:hp-main} has been used to derive estimates for some
counting problems of a diophantine nature. In this section we
illustrate our result by improving the estimates presented in these
papers.

In this section we assume that $K=\C$, and the differentiation
operator $D$ is chosen such that the field of constants is $k=\bar\Q$.

\subsection{Transcendental points in subvarieties of semi-abelian varieties}
\label{sec:hp-app}

Recall that a semi-abelian variety is an extension of an abelian variety by
a torus. Let $A$ be a semi-abelian variety and $\Gamma\subset A$ as subgroup
of finite rational rank $r:=\dim_\Q \Gamma\otimes\Q$. Finally let $X$ be a subvariety
of $A$.

In \cite{hp:effective} effective bounds are given on the number of
points in the intersection $X\cap\Gamma$ under various conditions on
$\Gamma,X$. The estimates are presented in terms of the following
data. Suppose that $A$ has dimension $n$ and is defined over $k$. We
assume that $A$ is embedded as a locally closed subset of a projective
space $\C P^N$. Let $\w_1,\ldots,\w_n$ be a basis of
translation-invariant differential forms on $A$.

We assume that $A$ is covered by $t$ affine charts, and that each
$\w_i$ is given in each chart by a polynomial in the local coordinates
$x_1,\ldots,x_N$ and $\d x_1,\ldots,\d x_n$. Let $d_A$ be the maximal
degree of the equations defining $A$ in any of the charts, and $d_\w$
be the maximal degree of any of the polynomials defining
$\w_1,\ldots,\w_n$ in any of the charts. Finally let $d_X$ denote the
maximal degree of the equations defining $X$ in any of the charts.
We assume for simplicity of the formulation that $d_X\ge d_A$.

The following is a refined form of the main result of
\cite{hp:effective}.

\begin{Thm}[cf. \protect{\cite[Theorem 1.1]{hp:effective}}]
  \label{thm:hp-main-imp}
  Suppose $A,X$ are defined over $k$ and there exist no positive-dimensional
  subvarieties $X_1,X_2\subset A$ such that $X_1+X_2\subset X$. Then
  \begin{equation}
    \#\left[(X\cap\Gamma)\setminus X(k)\right] \le F_{N,n,r} \cdot t\cdot d_\w^{Nr} d_x^n,
  \end{equation}
  where
  \begin{equation}
    F_{N,n,r} =\frac{E_{N(r+1),N-n}}{N(r+1)!} \binom{Nr+n}{n} d_A^{N-n} 2^n.
  \end{equation}
\end{Thm}
\begin{proof}
  In \cite{hp:effective} it is shown that $\Gamma$ is contained in a
  finite-dimensional differential algebraic subgroup $G$ of $A$ such that
  $B:=(X\cap G)\setminus X(k)$ is finite. It is therefore enough to bound the
  number of points in $B$. Moreover, it is shown that in each of the
  affine charts on $A$, the group $G$ can be written in the form
  \begin{equation*}
    G = \{ x : x\^r\in S \}
  \end{equation*}
  where $S\subset (K^N)\^r$ is a variety defined by equations of degrees
  bounded by $d_\w$.

  We will apply Theorem~\ref{thm:main} with parameters
  \begin{align*}
    M &= K^N,\\
    l&=r,\\
    \Omega\^r&=M\^r\setminus\{\xi\^1=0\},\\
    Y &= \Omega\^r\cap S\cap \pi^{-1}(X).
  \end{align*}
  Since $Y$ is contained in $\pi^{-1}(A)$ which has codimension $N-n$
  and is cut out by equations of degree at most $d_A$, we can apply
  Theorem~\ref{thm:main} with
  \begin{equation}
    \Delta_1=\cdots=\Delta_{N-n}=d_A\Delta_\xi, \qquad \Delta = d_X\Delta_\xi+d_\w\Delta_\xi\^r.
  \end{equation}
  We note that $\Delta\subset\Delta'+\Delta''$ where $\Delta'=2d_X\Delta_\xi$ and $\Delta''$ is the
  polytope of polynomials of degree $d_\w$ in $\xi\^1,\ldots,\xi\^r$.

  Finally, $Y$ admits finitely many solutions, and by
  Theorem~\ref{thm:main} we have
  \begin{align*}
    \cN(Y) &\le E_{s,N-n} V(\Delta_1,\ldots,\Delta_{N-n},\Delta'+\Delta'',\ldots,\Delta'+\Delta'') \\
    &= E_{s,N-n} \binom{s-N+n}{n} \mv(\underbrace{d_A\Delta_\xi}_{N-n\text{ times}},
    \underbrace{\Delta'}_{n\text{ times}}, \underbrace{\Delta''}_{s-N\text{ times}}) \\
    &\le E_{s,N-n} \binom{s-N+n}{n} d_A^{N-n} (2d_X)^n d_\w^{N-n} \mv(\Delta_\xi\^l,\ldots,\Delta_\xi\^l) \\
    &= E_{s,N-n} \binom{s-N+n}{n} 2^n (s!)^{-1} d_A^{N-n} d_\w^{s-N} d_X^n 
  \end{align*}
  where in the second equality we expand the mixed volume by
  multi-linearity and use the fact that the value is non-zero if and
  only if $\Delta'$ appears exactly $n$ times.
\end{proof}

We note in particular that the bound of Theorem~\ref{thm:hp-main-imp}
is singly exponential with respect to $N$ and $r$, and has the natural
asymptotic $d_X^n$ with respect to $d_X$. We also remark that since
$A$ is smooth, the arguments of this paper could have been applied
with small changes in the ambient space $A$ rather than $K^N$, leading
to somewhat better estimates.

We next present a version of Theorem~\ref{thm:hp-main-imp} for a
torus. This result may be seen as an analog of the Kushnirenko
theorem. Similar analogs of the BKK theorem involving mixed volumes
can be obtained in a similar manner.

\begin{Thm}
  \label{thm:hp-torus}
  Let $A=(\C^*)^n$ and let $X\subset A$ be defined over $k$ by
  equations with Newton polytope $\Delta$. Suppose that there exist no
  positive-dimensional subvarieties $X_1,X_2\subset A$ such that
  $X_1+X_2\subset X$. Then
  \begin{equation}
    \#\left[(X\cap\Gamma)\setminus X(k)\right] \le F_{2n,n,r} 2^{n(2r+1)} \vol(\Delta)
  \end{equation}
\end{Thm}
\begin{proof}
  We embed $A$ in $\C^{2n}=\C[x_1,\ldots,x_n,y_1\ldots,y_n]$ with the
  equations $x_iy_i=1$. Then the invariant forms $\d x_i/x_i$ become
  polynomials of degree $2$, so $d_A=d_\w=2$. The proof proceeds in a
  manner analogous to the proof of Theorem~\ref{thm:hp-main-imp}.
\end{proof}

A special case with $n=2$ in Theorem~\ref{thm:hp-torus} is of some
particular interest. It has been considered by Buium
\cite{buium:torus} who gave an iterated-exponential bound, answering a
question posed by Bombieri. The bound was improved in
\cite{hp:effective} to a doubly-exponential bound as follows.

\begin{Thm}[\protect{\cite[Corollary~1.2]{hp:effective}}] \label{thm:hp-torus-dim2}
  Let $f\in\bar\Q[x,y]$ be an irreducible polynomial of degree $d$,
  whose zero-locus in $(\C^*)^2$ is not a translate of a torus. Let
  $\Gamma\subset\C^*$ be a subgroup of rational rank $r$. Then
  \begin{equation}
    \#\{ (x,y)\in\Gamma^2 : f(x,y)=0, x\not\in\bar\Q \} \le d^{r 2^r} (r+1)^{2(2^r+1)}
  \end{equation}
\end{Thm}
\noindent From Theorem~\ref{thm:hp-torus} we obtain
\begin{Thm} \label{thm:hp-torus-dim2-improved} Let $f\in\bar\Q[x,y]$
  be an irreducible polynomial with Newton polygon $\Delta$, whose
  zero-locus in $(\C^*)^2$ is not a translate of a torus. Let
  $\Gamma\subset\C^*$ be a subgroup of rational rank $r$. Then
  \begin{equation}
    \#\{ (x,y)\in\Gamma^2 : f(x,y)=0, x\not\in\bar\Q \} \le F_{4,2,r} 2^{4r+2} \vol(\Delta).
  \end{equation}
\end{Thm}

\noindent We note in particular that the bound of
Theorem~\ref{thm:hp-torus-dim2} is singly-exponential in $r$, and if
we assume that $\deg f=d$ then $\vol(\Delta)=d^2/2$.

We also present a refined form of \cite[Theorem 1.1]{hp:effective},
replacing the condition on $X$ in Theorem~\ref{thm:hp-torus} by a
condition on the lattice $\Gamma$. For simplicity we present the result in
the context of the torus, although the proof works in general in the
same way as the proof of~\ref{thm:hp-main-imp}.

\begin{Thm}
  \label{thm:hp-torus2}
  Let $A=(\C^*)^n$ and let $X\subset A$ be defined by equations with
  Newton polytope $\Delta$, and suppose that it does not contain a
  translate of a non-trivial subtorus. Suppose that
  $\Gamma\cap A(k)=\{0\}$. Then
  \begin{equation}
    \#(X\cap\Gamma) \le n(2r+1) F_{2n,n,r} 2^{n(2r+1)} \vol(\Delta)
  \end{equation}
\end{Thm}

\begin{proof}
  In \cite[Lemma~6.1]{hp:effective} it is shown that $\Gamma$ is
  contained in a finite-dimensional differential algebraic subgroup
  $G$ of $(K^*)^n$ such that $X\cap G$ intersects only finitely many
  translates $c_1(k^*)^n,\ldots,c_\nu(k^*)^n$ of $(k^*)^n$. By the
  assumption on $\Gamma$, each such translate intersects $\Gamma$ at
  most once, so it will suffice to give a bound on $\nu$.

  We choose $M,\Omega\^r$ as in the proof of
  Theorem~\ref{thm:hp-main-imp} and write $X\cap G$ as the set of
  solutions of a variety $Y\subset\Omega\^r$. The logarithmic
  derivative $l(x)=(Dx)/x$ takes a different value $l(c_j)$ on each of
  the translates $c_j(k^*)^n$, so the solutions of $Y$ are contained
  in a union of $\nu$ disjoint planes $\xi\^1=c_j\xi$ in $\Omega\^r$
  (and meets each of them). Then the same is true for the Zariski
  closure $\red Y$, and it is thus enough to bound the degree of this
  variety. The result now follows by Corollary~\ref{cor:main-deg}.
\end{proof}

\subsection{Isogeny classes of elliptic curves}
\label{sec:isog-app}

Denote by $E_x$ the elliptic curve with $j$-invariant $x\in\C$. We
write $E_x\sim E_y$ is $E_x$ is isogenous to $E_y$. We denote
\begin{equation}
  \iso(x) = \{y\in\C: E_y\sim E_x\}
\end{equation}
and more generally, for $\bar x\in\C^n$,
\begin{equation}
  \iso(\bar x) := \{ \bar y \in\C^n : E_{x_j}\sim E_{y_j} \text{ for } j=1,\ldots,n \}.
\end{equation}

In \cite{fs:j-func} the authors, following a question of Mazur,
consider the following problem: given an automorphism $\alpha$ of
$\C P^1$ and $\tau\in\C$, estimate the number of elliptic curves
$E_\rho$ satisfying $E_\rho\sim E_\tau$ and
$E_{\alpha\cdot\rho}\sim E_{\alpha\cdot\tau}$. In particular they
obtain the following estimate.

\begin{Thm}[\protect{\cite[Section~6.1]{fs:j-func}}] \label{thm:fs}
  Assume $\tau$ is transcendental and set
  \begin{equation} \label{eq:isog-S}
    S := \{ (z,w) : w=\alpha\cdot z \} \cap \iso(\tau,\alpha\cdot\tau).
  \end{equation}
  Then the size of $S$, i.e. the number of elliptic curves $E_\rho$
  satisfying $E_\rho\sim E_\tau$ and
  $E_{\alpha\cdot\rho}\sim E_{\alpha\cdot\tau}$, is at
  most\footnote{\label{footnote:fs}The constant appearing in
    \cite{fs:j-func} contains a minor computational error, this is a
    tentative correction.} $2^{24} 36^7\simeq10^{18}$.
\end{Thm}

The proof of this theorem, as in~\secref{sec:hp-app}, is based on
the introduction of a differential algebraic construction. Namely,
recall that the Schwratzian operator is defined by
\begin{equation}
  S(x) = \left(\frac{x''}{x'}\right)' - \frac12 \left(\frac{x''}{x'}\right)^2
\end{equation}
where we interpret this as an operator on our differentially closed
field $\C$, and write $x'$ as a shorthand for $Dx$. We introduce the
differential operator
\begin{equation}
  \chi(x) = S(x) + R(x) (x')^2, \qquad R(x) = \frac{x^2-1968x+2654208}{2x^2(x-1728)^2}
\end{equation}
which is a third order algebraic differential operator vanishing on
Klein's j-invariant $j$ \cite{buium:moduli}.

Let $\tau\in\C$ be transcendental (i.e., non-constant with respect to
our chosen differential structure). By a theorem of Buium
\cite{buium:moduli}, the set $\chi^{-1}(\chi(\tau))$ is the Kolchin
closure of $\iso(\tau)$. One may therefore attempt to study the set
$S$ in~\eqref{eq:isog-S} by considering the (possibly larger) set
\begin{equation}
  Z := \{ (z,w) : w=\alpha\cdot z \} \cap
  \{(z,w): \chi(z)=\chi(\tau),\chi(w)=\chi(\alpha\cdot\tau)\}.
\end{equation}
Of course, even if the set $S$ is finite (which is not an obvious
assertion), it is not clear a-priori that the set $Z$ must also be
finite. However, in \cite[Section 6.1]{fs:j-func} it is proven that
this is indeed the case. Since $Z$ is a set given by differential
algebraic conditions, it is thus possible to estimate its size using
the methods of \cite{hp:effective}, and this is carried out in
\cite{fs:j-func} to give Theorem~\ref{thm:fs}.

We now apply our results, specifically
Theorem~\ref{thm:comp-intersections}, to estimate the size of $Z$. We
begin by expressing $Z$ as the set of solutions of a variety of
differential conditions $Y$. Let $\alpha\cdot z=\frac{az+b}{cz+d}$. We
choose $M=\C^2,l=3$ and let $\oml$ be the open dense subset of $M\^l$
obtained by removing the polar divisor of $\alpha\cdot\xi$ (i.e.
$\{c\xi+d=0\}$) as well as the polar divisors of $\chi(\xi)$ and
$\chi(\eta)$.

We write \emph{six} explicit equations for $Y$. The first is given by
$P_1 : (c\xi+d)\eta=a\xi+b$. We obtain the next three equations
$P_{2\ldots4}$ by taking the first three derivatives of $P_1$ with
respect to $D$ (and replacing $D\xi$ by $\xi\^1$, etc.). One easily
checks that these equations define a complete intersection (in fact,
they define the third prolongation of the graph of $\alpha$: at this
point it is essential that we removed the polar divisor $c\xi+d=0$).
Clearly $\Delta(P_j)\subset\Delta_\xi\^3+\Delta_\eta\^3$ for
$j=1,2,3,4$.

The next two equations $P_5,P_6$ are given by $\chi(\xi)=\chi(\tau)$
and $\chi(\eta)=\chi(\alpha\cdot\tau)$ respectively, where we clear
out all the denominators. We have $\Delta(P_5)\subset 6\Delta_\xi\^3$
and $\Delta(P_6)\subset 6\Delta_\eta\^3$. Since these equations are
linear in $\xi\^3$ and $\eta\^3$ respectively, it is not hard to
check that they are each irreducible (this was already noted in
\cite[Section~5.2]{fs:j-func}). It follows that $P_{1\ldots6}$ define
a complete intersection. Indeed, otherwise $P_1,\ldots,P_5$ would
imply $P_6$. But $P_5$, being an equation of order 3, admits
infinitely many solutions $z$, and for each of these
$(z,\alpha\cdot z)$ would be a solution of $P_{1,\ldots,5}$ and hence
also of $P_6$, contradicting the fact that $Y$ has finitely many
solutions.

Finally, we apply Theorem~\ref{thm:comp-intersections} to $Y$. In
computing $\Gamma$ we also take into account Remark~\ref{rem:refined-gamma},
\begin{equation}
  \Gamma \subset \big[(8-6+1)+4+6\big](\Delta_\xi\^3+\Delta_\eta\^3)=13(\Delta_\xi\^3+\Delta_\eta\^3)
\end{equation}
and
\begin{align*}
  \cN(Y) &\le 4^3 \cdot 8! V(\underbrace{\Delta_\xi\^3+\Delta_\eta\^3}_{4\text{ times}},
    6\Delta_\xi\^3, 6\Delta_\eta\^3,
    \underbrace{13(\Delta_\xi\^3+\Delta_\eta\^3)}_{2\text{ times}}) \\
  &\le 2^6 6^2 13^2 \cdot 8! V(\Delta_\xi\^3,\Delta_\eta\^3,
    \underbrace{\Delta_\xi\^3+\Delta_\eta\^3}_{6\text{ times}}) \\
  &= 2^6 6^2 13^2 \binom63 \cdot 8! V(
    \underbrace{\Delta_\xi\^3}_{4\text{ times}},
    \underbrace{\Delta_\eta\^3}_{4\text{ times}}) \\
  &= 2^6 6^2 13^2 \binom63 = 2^{10} \cdot 3^3 \cdot 13^2.
\end{align*}
In conclusion,

\begin{Thm}[cf. Theorem~\ref{thm:fs}] \label{thm:fs-improved}
  Assume $\tau$ is transcendental. Then the number of elliptic curves
  $E_\rho$ satisfying $E_\rho\sim E_\tau$ and
  $E_{\alpha\cdot\rho}\sim E_{\alpha\cdot\tau}$, is at most
  $2^{10} \cdot 3^3 \cdot 13^2\simeq5\times10^6$.
\end{Thm}

\begin{Rem}
  One could have derived a bound directly using Theorem~\ref{thm:main}
  without the derivation of the extra equations $P_{2,\ldots,4}$ and
  the fact that the intersection with $P_5$ and $P_6$ is complete. We
  presented this more detailed approach because it gives a somewhat
  better estimate, and also to illustrate a computation involving
  mixed volumes in Theorem~\ref{thm:comp-intersections}.
\end{Rem}

The computation above could certainly be improved somewhat: by using
the precise Newton polytopes of $P_{1,\ldots,6}$; by accurately
computing the resulting mixed volumes; and by following the proof of
Theorem~\ref{thm:comp-intersections} where various inaccurate
estimates were invariably made.

Generalizing to varieties of arbitrary dimension and degree,
\cite{fs:j-func} gives the following result.

\begin{Cor}[\protect{\cite[Corollary~6.9]{fs:j-func}}] \label{cor:fs}
  Let $V\subset\C^n$ be a Zariski closed set of dimension $m$ and
  $\bar\tau\in\C^n$ an $n$-tuple of transcendental numbers. Let $W$
  denote the Zariski closure of $V\cap\iso(\bar\tau)$. Then $W$ is a
  \emph{weakly-special} variety, and\textsuperscript{\ref{footnote:fs}} 
  \begin{equation}
     \deg W\le (2^n\deg V)^{3\cdot2^{3m}} 6^{2^{3m}-1}.
  \end{equation}
\end{Cor}

We refer the reader to \cite{fs:j-func} for the definition of a weakly
special variety. As for the degree estimate, the proof of this
corollary proceeds in a manner analogous to the proof of
Theorem~\ref{thm:fs}. Arguing in a manner analogous to the proof of
Theorem~\ref{thm:fs-improved} and using Corollary~\ref{cor:main-deg}
we obtain the following result.

\begin{Cor}
  Let $V\subset\C^n$ be a Zariski closed set defined by equations of
  degree $d$, and $\bar\tau\in\C^n$ an $n$-tuple of transcendental
  numbers. Let $W$ denote the Zariski closure of
  $V\cap\iso(\bar\tau)$. Then $\deg W \le G_n d^n$ where $G_n$ is an
  explicit constant, singly-exponential in $n$.
\end{Cor}

In conclusion we remark the estimates in terms of degrees in this
section could be generalized to estimates in terms of volumes of
Newton polytopes, with the proofs extending verbatim.

\bibliographystyle{plain} \bibliography{refs}

\end{document}